\documentclass[12pt]{amsart}
\usepackage{amsmath, amsthm, amsfonts, amssymb, color}
\usepackage[ colorlinks, plainpages]{hyperref}

\newtheorem{theorem}{Theorem}[section]

\newtheorem{lemma}[theorem]{Lemma}
\theoremstyle{definition}

\newtheorem{notation}[theorem]{Notation}

\begin{document}
	
\title[ON TWO CONJECTURES ABOUT THE SUM OF ELEMENT ORDERS]
{On Two Conjectures about the Sum of Element Orders}	
\author[  M. Baniasad, B. Khosravi ]{ Morteza Baniasad Azad \& Behrooz Khosravi }
\address{ Dept. of Pure  Math.,  Faculty  of Math. and Computer Sci. \\
Amirkabir University of Technology (Tehran Polytechnic)\\ 424,
Hafez Ave., Tehran 15914, Iran \newline }
\email{ baniasad84@gmail.com}
\email{ khosravibbb@yahoo.com}

\thanks{}


\begin{abstract}
		Let $G$ be a finite group and  $\psi(G) = \sum_{g \in G} o(g)$, where $o(g)$ denotes
		the order of $g \in G$.
		
		First, we prove that 
		if $G$ is a group of order $n$ and $\psi(G) >31\psi(C_n)/77$, where $C_n$ is the cyclic group of order $n$, then $G$ is supersolvable.
	     This proves a conjecture of M.~{T\u{a}rn\u{a}uceanu}. 
	     
       Moreover, 	     M.~{Herzog}, P.~{Longobardi} and  M.~{Maj}
	     put forward the following conjecture:
	     If $H\leq G$,
	     then $\psi(G) \leqslant \psi(H) |G:H|^2$.
	      In the sequel,
	      by an example we show that this conjecture is not satisfied in general.
\end{abstract}

\maketitle
\section{\bf Introduction}
In this paper all groups are finite. The cyclic group of order $n$ is denoted by $C_n$.  Let $ \psi(G) = \sum_{g \in G} o(g)$,  the sum of element orders in a  group $G$.
The function $\psi(G)$ was introduced  by Amiri, Jafarian
and  Isaacs \cite{amiri2009sums}.

Many authors try to get some relations between the structure of the group 
$G$ and $\psi(G)$
(see \cite{herzogsurvey,herzogComm,herzogNew,Jafarian}).
 Recently it is proved that if 
$\psi(G)> \frac{211}{1617} \psi(C_n)$, where $G$ is a finite group of order $n$, then $G$ is solvable \cite{Baniasad,herzog2018two}.

From the observation that $A_4$ satisfies $\psi(A_4) =31$ and $\psi(C_{12}) =77$, in \cite{TzbMATH06233336} M.~{T\u{a}rn\u{a}uceanu} put forward the following conjecture:
\\
\\
\textbf{[T]-Conjecture} \cite[Conjecture 1.5]{TzbMATH06233336} \textit{If $G$ is a group of order $n$ and
	$\psi(G) > \frac{31}{77} \psi(C_n)$,
	then $G$ is supersolvable.}
\\

 First, in this paper we prove the validity of this conjecture.
Nevertheless, for groups of
odd order, it is possible to prove a stronger result:
\begin{theorem} \label{oddsuper}
	Suppose that $n=|G|$ is odd.
	If $\psi(G) > \frac{271}{3647} \psi(C_n)$, then $G$ is supersolvable.
\end{theorem}


On the other hand, in \cite{herzog2018two} there exist three conjectures about $\psi(G)$.
The validity of Conjecture 5 was proved in \cite{Baniasad}
and
the validity of Conjecture 6 was proved in \cite{Bahri}.
Moreover the following conjecture was posed in that paper.
\\
\\
\textbf{[HLM]-Conjecture} \cite[Conjecture 7]{herzog2018two} \textit{If $H\leq G$,
	then $\psi(G) \leqslant \psi(H) |G:H|^2$.}
\\

Finally, in the sequel of this paper we give an example which shows that this conjecture is not satisfied in general and it seems that we need extra conditions on $G$ or $H$ to have this result.

 For the proof of these results, we need the following
 lemmas.
\begin{lemma}\cite[Corollary B]{amiri2009sums} \label{sumsylow}
	Let $P \in {\rm Syl}_p(G),$ and assume that $P \unlhd G$ and that $P$ is cyclic. Then
	$\psi(G) \leq \psi(P)\psi(G/P)$, with equality if and only if $P$ is central in $G$.
\end{lemma}
\begin{lemma}
	\cite[Proposition 2.6]{herzog2018two} \label{prop}
	Let $H$ be a normal subgroup of the finite group $G$. Then
	$\psi(G) \leq \psi(G/H)|H|^2$.
\end{lemma}
\begin{lemma}\cite[Lemma 2.1]{Amiri2011zbMATH05906990} \label{sumdirect}
	If $G$ and $H$ are finite groups, then $\psi(G \times H) \leq \psi(G)\psi(H)$. Also,
	$\psi(G \times H) = \psi(G)\psi(H)$ if and only if $\gcd(|G|, |H|) = 1$.
\end{lemma}
\begin{lemma}\cite[Proposition 2.5]{herzog2017exact} \label{2p}
Let $G$ be a finite group and suppose that there exists $x\in G$ such that $|G:\langle x\rangle|<2p$,
where $p$ is the maximal prime divisor of $|G|$. Then one of the following holds:

(i) $G$ has a normal cyclic Sylow $p$-subgroup,

(ii) $G$ is solvable and $\langle x\rangle$
is a maximal subgroup of $G$ of index either $p$ or $p+1$.
\end{lemma}
\begin{lemma}\cite[Theorem 2.20]{isaa} (Lucchini) \label{lucchini}
	Let $A$ be a cyclic proper subgroup of a finite
	group $G$, and let $K={\rm core}_G(A)$. Then $| A : K | < | G : A |$, and in particular,
	if $|A| > |G : A|$ , then $K>1$.
\end{lemma}
\begin{lemma}\cite[Lemma 2.1]{Baniasad}\label{aval}
	Let $G$ be a group of order $n={p_1}^{\alpha_1}\cdots{p_k}^{\alpha_k}$, where ${p_1}, \cdots,{p_k}$ are distinct primes. Let $\psi(G)>\dfrac{r}{s} \psi(C_n)$, for some integers $r, s$. Then  there exists a cyclic subgroup $ \langle{x}\rangle$ such that
	\[ [G:\langle{x}\rangle]< \dfrac{s}{r} \cdot \dfrac{p_1+1}{p_1}\cdots \dfrac{p_k+1}{p_k}. \]
\end{lemma}
\begin{lemma} \cite[Lemmas 2.4 and 2.5]{herzog2018two} \label{herzog}
Let  $n = {p_1}^{\alpha_1} {p_2}^{\alpha_2} \cdots {p_r}^{\alpha_r}$  be a positive integer,
 where $p_{i}$ are primes, $p_1 < p_2 < \cdots < p_r = p$ and $\alpha_i>0$, for each $1 \leq i \leq r$.
 If $p \geq 11$, then
	$$\psi(C_n) \geq \dfrac{385}{96} \dfrac{n^2}{p+1}.$$
\end{lemma}

\section{\bf  Supersolvability}
\subsection{Proof of [T]-Conjecture}
	 Let $G$ be a group of order $n = {p_1}^{\alpha_1} {p_2}^{\alpha_2} \cdots {p_r}^{\alpha_r}$,
	where $p_{1}, p_{2}, \dots , p_{r}$ are primes, $p_1 < p_2 < \cdots < p_r = p$ such that $\alpha_i>0$, for each $1 \leq i \leq r$.
	 By assumption,
	 $\frac{31}{77}\psi(C_n)<\psi(G)$.
	 By induction on $|\pi(G)|$, we prove that $G$ is a supersolvable group.
	 
	 If $|\pi(G)|=1$, then $G$ is a $p$-group and so $G$ is supersolvable.
Assume that $|\pi(G)|\geq 2$ and the theorem holds for each group $H$ such that $|\pi(H)|<|\pi(G)|$. Now we consider the following two cases:

\textbf{Case (I)}
 If $G$ has a  normal cyclic Sylow subgroup $Q$, then
by Lemma \ref{sumsylow}  we have
$\psi(G) \leq \psi(Q)\psi(G/Q)$. Using Lemma \ref{sumdirect} and the assumptions we have
\[\dfrac{31}{77} \psi(C_{|G/Q|})\psi(C_{|Q|})=\dfrac{31}{77} \psi(C_n)<\psi(G)\leq\psi(Q)\psi(G/Q)=\psi(C_{|Q|})\psi(G/Q).\]
Therefore $\frac{31}{77}\psi(C_{|G/Q|})<\psi(G/Q)$ and $|\pi(G/Q)|<|\pi(G)|$. By the inductive hypothesis, $G/Q$ is supersolvable
and so $G$ is a supersolvable group.

\textbf{Case (II)} Let $G$ have no normal cyclic Sylow  subgroup.

If  $p\geq11$, then  by Lemma \ref{herzog}, we have 
\[\psi(G) > \dfrac{31}{77}\psi(C_n) \geq \dfrac{31}{77} \cdot \dfrac{385}{96} \dfrac{n^2}{p+1}. \]
Thus	there exists $x \in G$ such that $o(x)>\frac{31}{77} \cdot\frac{385}{96} \frac{n}{p+1}$. Therefore
\[[G:\langle x \rangle] < \dfrac{77}{31} \cdot \dfrac{96}{385} (p+1) \leq \dfrac{77}{31} \cdot \dfrac{96}{385} \cdot \dfrac{12}{11} p< 0.68p<p, \]
which is a  contradiction,  by Lemma \ref{2p}.
Therefore $\pi(G) \subseteq \{2, 3, 5, 7\}$, where $ 2 \leq|\pi(G)| \leq 4$. 
Now we consider the following cases:
\\
\textbf{Case 1.} Let $\pi(G)=\{2, 3\}$. Then $|G|=2^{\alpha_1}3^{\alpha_2}$. In this case we have
\begin{align} \label{sum23}
\psi(G)>\dfrac{31}{77}\psi(C_{|G|})>\dfrac{31}{77}\cdot \dfrac{2^{2{\alpha_1}+1}}{2+1} \cdot
\dfrac{3^{2{\alpha_2}+1}}{3+1} =\dfrac{31}{77} \cdot \dfrac{1}{2}n^2.
\end{align}
It follows  that there exists $x\in G $ such that $o(x)>\frac{31}{77} \cdot \frac{1}{2}n$. We conclude that
$|G:\langle x \rangle|<\frac{77}{31} \cdot  2 < 4.97$. By Lemma \ref{2p}, we have $[G:\langle x \rangle]=3$ or $4$.  Let $ H={\rm core}_G(\langle x \rangle)$.
	\begin{itemize}
	\item Let $[G:\langle x \rangle]=3$. 
	 By Lemma \ref{lucchini}, $[\langle x \rangle : H] < [G : \langle x \rangle] = 3$.
	Therefore $G/H$ is  supersolvable and so $G$ is a supersolvable group.
	\item Let $[G:\langle x \rangle]=4$.
	 By Lemma \ref{lucchini}, $[\langle x \rangle : H] < [G : \langle x \rangle] = 4$. 
	If $G/H$ is a supersolvable group, then we get the result. Let $G/H$ be non-supersolvable. Therefore $ [\langle{x}\rangle: H] =  3$,  $|G/H|=12$ and so $G/H\cong A_4$.
	\\	
	$\blacktriangleright$
	If $2$ divides $|H|$, then there
	exists
	a characteristic subgroup $M$ in $H$ such that $|H:M|=2$.
	Thus $G/M$ is a non-supersolvable group of order $24$. Therefore using the list of
	such groups ($\rm{SL}(2,3), S_4, C_2 \times A_4$) and their
	$\psi$-values ($99, 67, 87$), we have $\psi(G/M) \leq 99$. By Lemma \ref{prop} we have
	$\psi(G) \leq \psi(G/M)|M|^2 \leq 99 (n/24)^2$.
	Using (\ref{sum23}), $	\psi(G)>\frac{31}{77}\cdot  \frac{1}{2} n^2$. Therefore
	\[\dfrac{31}{77}\cdot  \dfrac{1}{2} n^2< 99 \dfrac{n^2}{24^2},\]
	which is a contradiction.	
	\\	
	$\blacktriangleright$
	If $2 \nmid |H|$, then $|G|=2^2 3^{\beta}$.
	Let $\beta \geq 2$. Then there exists a characteristic subgroup $M$ in $H$ such that $|H:M|=3$. Therefore 
	$M \unlhd G$ and $|G/M|=36$. 
	Thus $G/M$ is a non-supersolvable group of order $36$. Therefore using the list of
	such groups ($(C_2 \times C_2) : C_9, (C_3 \times C_3) : C_4, C_3 \times A_4$) and their
	$\psi$-values ($265, 115, 121$), we have $\psi(G/M) \leq 265$. By Lemma \ref{prop} we have
	$\psi(G) \leq \psi(G/M)|M|^2 \leq 265 (n/36)^2=265 \cdot 3^{2\beta-4}$.
	On the other hand, we have	$\psi(G) > \frac{31}{77} \psi(C_n)$. Thus
	$$\psi(G) > \frac{31}{77} \psi(C_4) \psi(C_{3^{\beta}})=\frac{31}{77} \cdot  11 \cdot \dfrac{3^{2\beta+1}+1}{4}. $$
	Therefore
	\begin{align*}
	\frac{31}{77} \cdot  11 \cdot \dfrac{3^{2\beta+1}+1}{4}<265 \cdot 3^{2\beta-4} &\Rightarrow
		\frac{31}{28}  \cdot 3^{2\beta+1}<265 \cdot 3^{2\beta-4} \\
		&\Rightarrow  31 \cdot 3^5 < 28 \cdot 265,
	\end{align*}
	which is a contradiction. Thus $\beta = 1$ and so $|H|=1$ and  $G\cong A_4$, which is a contradiction, since $\psi(A_4)=31$.	
	\end{itemize}
\textbf{Case 2.} Let $\pi(G)=\{2, 3, 5\}$. Then by Lemma \ref{aval},
there exists $x\in G $ such that 
$[G:\langle x \rangle]<\frac{77}{31} \cdot \frac{12}{5}<5.97$.
 By Lemma \ref{2p}, we have $[G:\langle x \rangle]=5$.
 By Lemma \ref{lucchini}, $[\langle x \rangle : {\rm core}_G(\langle x \rangle)] < [G : \langle x \rangle] = 5$.
 We have $|G/{\rm core}_G(\langle x \rangle)|= 5, 10, 15$ or $20$. Therefore
 $G/{\rm core}_G(\langle x \rangle)$
  is  supersolvable and so $G$ is a supersolvable group.
\\
\textbf{Case 3.} Let $\pi(G)=\{2, 3, 5, 7\}$. Then by Lemma \ref{aval},
	there exists $x\in G $ such that 
	$[G:\langle x \rangle]<\frac{77}{31} \cdot \frac{96}{35}<6.9<7$,
		which is impossible.
		\\
\textbf{Case 4.} Let $\pi(G)=\{2, 5\}$ or $\pi(G)=\{ 3, 5\}$. Then by Lemma \ref{aval}, there exists $x \in G$ such that
 $[G:\langle x \rangle]<5$, which is a contradiction \\
\textbf{Case 5.} Let $\pi(G)=\{2, 7\}$, $\pi(G)=\{3, 7\}$, $\pi(G)=\{5, 7\}$, $\pi(G)=\{2, 3, 7\}$ or $\pi(G)=\{2, 5, 7\}$. Then by Lemma \ref{aval}, there exists $x \in G$ such that
$[G:\langle x \rangle]<7$ and we get a contradiction. \\
 The proof is now complete.

\subsection{\bf Equality condition of [T]-Conjecture}
\begin{lemma} \label{bahh}
	Let $G$ be a non-supersolvable group of order $n$ and $p \mid n$. Furthermore,	 let
	 $\psi(G) =\frac{31}{77}\psi(C_n)$ and $P \in {\rm Syl}_p(G)$ is cyclic and normal in $G$. Then $P$ has a normal $p$-complement $K$ in $G$, where $\psi(K) = \frac{31}{77} \psi(C_{|K|})$.
\end{lemma}
\begin{proof}
 By Lemma \ref{sumsylow}, we have
\begin{align} \label{tttttt}
 \psi(P )\psi(G/P )\geq
 \psi(G) = \dfrac{31}{77}
\psi(C_n) = \dfrac{31}{77}
\psi(C_{|P|})\psi(C_{|G/P|}).\end{align}
Therefore
\[ \psi(G/P ) \geq \dfrac{31}{77}
\psi(C_{|G/P|}).\]
Now by the non-supersolvability of G and the above result, we conclude that
 \[\psi(G/P ) =
\dfrac{31}{77}
\psi(C_{|G/P|}).\]
 It follows that the equality holds in (\ref{tttttt}). Thus $P \leq Z(G)$, by Lemma \ref{sumsylow}.
Therefore by Burnside’s normal $p$-complement theorem, there exists $K \unlhd G$ such that
$G = P \times K.$
\end{proof}

\begin{theorem}
	Suppose that $G$ is a non-supersolvable group of order $n$ such that
	$\psi(G) =\frac{31}{77}\psi(C_n)$.
	Then $G = A_4 \times C_m$, where $(6, m) = 1$.
\end{theorem}

\begin{proof}
	Let $p$ be the largest prime
	divisor of $n$. 	We prove the result by induction on $p$.
	By the assumption, $|\pi(G)| \geq 2$, and so $p \geq 3$.	
	\begin{enumerate}
		\item Let $p=3$.\\
		We note that $G$ has no normal cyclic Sylow subgroup.		
		There exists $x\in G $ such that 	$|G:\langle x \rangle|<\frac{77}{31} \cdot  2 < 4.97$. By Lemma \ref{2p}, we have $[G:\langle x \rangle]=3$ or $4$. 
		If $[G:\langle x \rangle]=3$, then $G$ is a supersolvable group, which is a contradiction.
		Therefore $[G:\langle x \rangle]=4$.
		We have $G/{\rm core}_G(\langle x \rangle)\cong A_4$, since  $G$ is  non-supersolvable.
		If ${\rm core}_G(\langle x \rangle)\neq 1$, then similarly to Case 1 in the proof of [T]-Conjecture we get a contradiction.
		Therefore ${\rm core}_G(\langle x \rangle) = 1$ and so  $G\cong A_4$.
		\item Let $p=5$.
		\begin{enumerate}
			\item Let $\pi(G)=\{2, 3, 5\}$.\\
			 Then by Lemma \ref{aval},
			there exists $x\in G $ such that 
			$[G:\langle x \rangle]<\frac{77}{31} \cdot \frac{12}{5}<5.97<10$.			
			By Lemma \ref{2p}, 
			$G$ has a normal cyclic Sylow $5$-subgroup or  $[G:\langle x \rangle]=5$.\\
			$\bullet$ Let $G$ has a normal cyclic Sylow $5$-subgroup, say $P_5$.	
			Using Lemma \ref{bahh}, there exists $K \unlhd G$ such that $G = P_5 \times K$ and
			$\psi(K) = \frac{31}{77} \psi(C_{|K|})$.
			We notice that $K$ is non-supersolvable and $\pi(K) = \{2, 3\}$. 
			By the above case, we have
			  $K = A_4$ and $G = A_4 \times P_5$.\\
			$\bullet$ Let $[G:\langle x \rangle]=5$. Then $|G/{\rm core}_G(\langle x \rangle)|= 5, 10, 15$ or $20$. Therefore
			$G/{\rm core}_G(\langle x \rangle)$
			is  supersolvable and so $G$ is a supersolvable group, which is a contradiction.
			\item 
			Let $\pi(G)=\{2, 5\}$ or $\pi(G)=\{ 3, 5\}$. Then  there exists $x \in G$ such that
			$[G:\langle x \rangle]<5$. 	By Lemma \ref{2p}, 
			$G$ has a normal cyclic Sylow $5$-subgroup. Thus $G$ is a supersolvable group, which is a contradiction.
		\end{enumerate}
	\item 	Let $p=7$. \\ Then  there exists $x \in G$ such that
	$[G:\langle x \rangle]<7$. 	By Lemma \ref{2p}, 
	$G$ has a normal cyclic Sylow $7$-subgroup.
	Using Lemma \ref{bahh}, there exists $K \unlhd G$ such that $G = P_7 \times K$ and
	$\psi(K) = \frac{31}{77} \psi(C_{|K|})$.
	We notice that $K$ is non-supersolvable and 
	$\pi(K) = \{2, 3\}$ or $\{2, 3, 5\} $, therefore by the above cases we have
	$G\cong P_7 \times A_4$ or $P_7 \times P_5 \times A_4$, respectively and we get the result.
	
	\item Let   $p\geq11$. Then  	there exists $x \in G$ such that 
	\[[G:\langle x \rangle] < \dfrac{77}{31} \cdot \dfrac{96}{385} (p+1) \leq \dfrac{77}{31} \cdot \dfrac{96}{385} \cdot \dfrac{12}{11} p< 0.68p<p, \]
	By Lemma \ref{2p}, 
	$G$ has a normal cyclic Sylow $p$-subgroup, say $P$.
	Using Lemma \ref{bahh}, there exists $K \unlhd G$ such that $G = P \times K$ and
	$\psi(K) = \frac{31}{77} \psi(C_{|K|})$.
	Then $K$ is non-supersolvable and 
	$\pi(K) = \pi(G) \setminus \{p\}$.
	Therefore using the above discussion and the induction hypothesis
	 we get the result.
	\end{enumerate}
\end{proof}
\subsection{Supersolvability for group of odd order}
\begin{notation}
	Let $ \{ q_1, q_2, q_3, \cdots \} $ be the set of \textbf{all} primes in an increasing order: 
	$2 = q_1 < q_2 < q_3 < \cdots $. Let also $q_0 = 1$.
	If $r, s$ are two  positive integers,
	we define the functions $f^{\prime}(r)$ and $h^{\prime}(s)$ as follows:
	\[f^{\prime}(1)=1,  f^{\prime}(r)=\prod_{i=2}^{r} \dfrac{q_i}{q_i+1}; \] 
	\[ h^{\prime}(2)=3,  h^{\prime}(s)=f^{\prime}(s-1)q_s.\]   
\end{notation}
In the  sequel,  let 
$A:= {\rm SmallGroup}(75,2)=(C_5 \times C_5) : C_3$. We know that
$\psi(A)=271$ and $\psi(C_{75})=3647$. Similarly to the proof of Lemmas 2.4 and 2.5 in \cite{herzog2018two} we get that:
\begin{lemma} \label{herzoggg}
	Let  $n = {p_1}^{\alpha_1} {p_2}^{\alpha_2} \cdots {p_r}^{\alpha_r}$  be a positive integer,
	where $p_{i}$ are primes, $2<p_1 < p_2 < \cdots < p_r = p$ and $\alpha_i>0$, for each $1 \leq i \leq r$.
	If $p \geq 37=q_{12}$, then $$\psi(C_n) \geq h^{\prime}(12)  \dfrac{n^2}{p+1}.$$
\end{lemma}
The proof of Theorem \ref{oddsuper}  is very similar to the proof of [T]-conjecture, and so we
remove the details of the proof.
\begin{proof}[Proof of Theorem \ref{oddsuper}]
	We prove the result by induction on $|\pi(G)| $.
	The assertion is trivial if $G$ is a $p$-group, and 
	let the result holds for each group $H$, where $|\pi(H)|<|\pi(G)|$ (exactly similar to the proof of [T]-Conjecture).

	We consider the following two cases:
	
	\textbf {Case(I)}
	If $G$ has a  normal cyclic Sylow subgroup $Q$, then we get the result.
	

	\textbf{Case (II)} Let $G$ have no normal cyclic Sylow  subgroup.
	\\
	Let $p$ be the largest prime divisor of $|G|$.
	If  $p\geq37$, then  by Lemma \ref{herzoggg}, we have 
	\[\psi(G) > \frac{271}{3647}\psi(C_n) \geq \frac{271}{3647} \cdot h^{\prime}(12) \dfrac{n^2}{p+1}. \]
	Thus	there exists $x \in G$ such that $o(x)>\frac{271}{3647} \cdot h^{\prime}(12) \frac{n}{p+1}$. Therefore
	\[[G:\langle x \rangle] < \dfrac{3647}{271} \cdot \dfrac{1}{h^{\prime}(12)} (p+1) \leq
	\dfrac{3647}{271} \cdot \dfrac{1}{h^{\prime}(12)} \cdot \dfrac{38}{37} p <p, \]
	which is a contradiction.

	Therefore $\pi(G) \subseteq \{3, 5, 7, \cdots, 31\}$ and so  $ 2 \leq|\pi(G)| \leq 10$. 
	Now we consider the following cases:
	\\
	\textbf{Case 1.}
	Let $p=31$, i.e. $\{31\} \subseteq \pi(G)  \subseteq \{3, 5, 7, \cdots, 31\}$. 
	
	Using Lemma \ref{aval}, there exists $x \in G$ such that
	$[G:\langle x \rangle] <  2p=62$. By Lemma \ref{2p}, we have $[G:\langle x \rangle]=31$. 
	Using Lemma \ref{lucchini}, $[\langle x \rangle : {\rm core}_G(\langle x \rangle)] < [G : \langle x \rangle] = 31$.
	We see that $G/{\rm core}_G(\langle x \rangle)$ is  supersolvable and so $G$ is a supersolvable group.
	\\
	\textbf{Case 2.} Let $p \in \{17, 19, 23, 29\}$. Similarly to  Case 1 we obtain that
	$G$ is a supersolvable group.
	
	Therefore $\pi(G) \subseteq \{3, 5, 7, 11, 13\}$, where $ 2 \leq|\pi(G)| \leq 5$.
	\\
	\textbf{Case 3.} Let $p=13$. 
	If $5\nmid |G|$, then
	by Lemma \ref{aval},
	there exists $x\in G $ such that 
	$[G:\langle x \rangle]<25$. Thus by Lemma \ref{2p} we have
	$[G:\langle x \rangle]=13$, and we get the result similarly to the above.
	If $5 \in \pi(G)$, 
	by Lemma \ref{aval},
	there exists $x\in G $ such that 
	$[G:\langle x \rangle]<29$, and by Lemma \ref{2p} we have
	$[G:\langle x \rangle]=13$ or $27$.
	
	If $[G:\langle x \rangle]=13$, then by Lemma \ref{lucchini} we have $|G/{\rm core}_G(\langle x \rangle)|=13m$, where
	$m \in \{3, 9, 5, 7, 11\}$. Therefore $G/{\rm core}_G(\langle x \rangle)$ is a supersolvable group and so $G$ is a supersolvable group.

	If $[G:\langle x \rangle]=27$, then 
	by Lemma \ref{lucchini} we have $|\langle x \rangle:{\rm core}_G(\langle x \rangle)|<27$.
	If $13 \mid |\langle x \rangle:{\rm core}_G(\langle x \rangle)| $, then $|G/{\rm core}_G(\langle x \rangle)|=13 \cdot 27$. Let $P_5 \in {\rm Syl}_5({\rm core}_G(\langle x \rangle))$. We have $P_5 \in {\rm Syl}_5(G)$ and $P_5 \unlhd G$, which is a contradiction.
	Therefore $13 \nmid |\langle x \rangle:{\rm core}_G(\langle x \rangle)| $.
	Let $P_{13} \in {\rm Syl}_{13}({\rm core}_G(\langle x \rangle))$. We have $P_{13} \in {\rm Syl}_{13}(G)$ and $P_{13} \unlhd G$, which is a contradiction.
	\\
	\textbf{Case 4.} Let $p=11$.\\
	$\bullet$ If $\pi(G)=\{3, 5, 7, 11\},$ then there exists 
	$x \in G$ such that
	$[G:\langle x \rangle]<26$. Therefore $[G:\langle x \rangle]=11$ or $25$. \\
	$\blacktriangleright$ If $[G:\langle x \rangle]=11$, then we have the result.\\
	$\blacktriangleright$ Let $[G:\langle x \rangle]=25$. Then $|G/{\rm core}_{G}\langle x \rangle|=25m$, where $m \in \{3, 5, 7, 9, 11, 15, 21\}$.
	If $m=11$, then we get the result. Otherwise, $G$ has a Sylow $11$-subgroup which is cyclic and normal, and this is impossible.
	\\
	$\bullet$ If $\pi(G) \subsetneq \{3, 5, 7, 11\}$, then there exists 
	$x \in G$ such that
	$[G:\langle x \rangle]<24$. Therefore $[G:\langle x \rangle]=11$ and so we have the result.\\
	\textbf{Case 5.} Let $p=7$.\\
	$\bullet$ If $\pi(G)=\{3, 7\}$ or $\{5, 7\}$, then easily we can see that $G$ is supersolvable.\\
	$\bullet$ If $\pi(G)=\{3, 5, 7\}$, then there exists 
	$x \in G$ such that
	$[G:\langle x \rangle]<25$. Therefore $[G:\langle x \rangle]=7, 15$ or $21$. \\
	$\blacktriangleright$ If $[G:\langle x \rangle]=7$, then we get the result.\\
	$\blacktriangleright$ Let $[G:\langle x \rangle]=15$. Then $|G/{\rm core}_{G}(\langle x \rangle)|=15m$, where $m \in \{3, 5, 7, 9\}$.
	If $m=7$, then we get the result. Otherwise, the Sylow $7$-subgroup of $G$  is cyclic and normal, which is impossible.\\
	$\blacktriangleright$ Let $[G:\langle x \rangle]=21$. Then $|G/{\rm core}_{G}\langle x \rangle|=21m$, where $m \in \{3, 5, 7, 9, 15\}$.
	If $m\in \{5, 15\}$, then we get the result. Otherwise,  the Sylow $5$-subgroup of $G$ is cyclic and normal, which is impossible.
	\\
	\textbf{Case 6.} Let $\pi(G)=\{3, 5\}$.  Then $|G|=3^{\alpha_1}5^{\alpha_2}$. In this case we have
	\begin{align} \label{sum35}
	\psi(G)>\frac{271}{3647}\psi(C_{|G|})>\frac{271}{3647}\cdot \dfrac{3^{2{\alpha_1}+1}}{3+1} \cdot
	\dfrac{5^{2{\alpha_2}+1}}{5+1} = \frac{271}{3647} \cdot \dfrac{5}{8}n^2.
	\end{align}
	Hence
	there exists $x\in G $ such that 
	$|G:\langle x \rangle|<22$. By Lemma \ref{2p}, we have $[G:\langle x \rangle]=5$ or $15$. 
	If $[G:\langle x \rangle]=5$, then we have the result. Let $[G:\langle x \rangle]=15$ and $ H={\rm core}_G(\langle x \rangle)$.
	By Lemma \ref{lucchini}, $[\langle x \rangle : H] < [G : \langle x \rangle] = 15$. 
	If $G/H$ is a supersolvable group, then we get the result. Let $G/H$ be non-supersolvable. Therefore $ [\langle{x}\rangle: H] =  5$ and $|G/H|=75$. Thus $G/H\cong A=(C_5 \times C_5) : C_3$.
	\\	
	$\blacktriangleright$
	If $3$ divides $|H|$, then there exists
	a characteristic subgroup
	$M$ in $H$ such that $|H:M|=3$. Therefore 
	$M \unlhd G$ and $|G/M|=225$. 
	Thus $G/H$ is a non-supersolvable group of order $225$. Therefore using the list of
	such groups ($(C_5 \times C_5) : C_9, C_3 \times ((C_5 \times C_5) : C_3)$) and their
	$\psi$-values ($2197, 1297$), we have $\psi(G/H) \leq 2197$. By Lemma \ref{prop} we have
	$\psi(G) \leq \psi(G/H)|H|^2 \leq 2197 (n/225)^2$.
	Using (\ref{sum35}), $	\psi(G)>\frac{271}{3647} \cdot  \frac{271}{3647} \cdot \dfrac{5}{8}n^2$. Therefore
	\[\frac{271}{3647} \cdot \dfrac{5}{8}n^2< 2197 (n/225)^2,\]
	which is a contradiction.	
	\\	
	$\blacktriangleright$
	If $3 \nmid |H|$, then $|G|=3 \cdot 5^{\beta}$.
	Let $\beta > 2$.
	Then there exists
	a characteristic subgroup
	$M$ in $H$ such that $|H:M|=5$.
	Therefore 
	$M \unlhd G$ and $|G/M|=375$. 
	Thus $G/H$ is a non-supersolvable group of order $375$. Therefore using the list of
	such groups ($((C_5 \times C_5) : C_5) : C_3, C_5 \times ((C_5 \times C_5) : C_3)$) and their
	$\psi$-values ($3771, 3771$), we have $\psi(G/H) = 3771$. By Lemma \ref{prop} we have
	$\psi(G) \leq \psi(G/H)|H|^2 \leq 3771 (n/375)^2$.
	Using (\ref{sum35}), we have 
	\[\frac{271}{3647} \cdot \dfrac{5}{8} < 3771 (n/375)^2.\]
	which is a contradiction. Thus $\beta = 2$ and so $G\cong A=(C_5 \times C_5) : C_3$, which is a contradiction.	
\end{proof}

\section{\bf  Counterexample for  [HLM]-Conjecture}

Using the notation in GAP, let $G={\rm SmallGroup}(32,7) \times C_m \cong ((C_8 : C_2) : C_2) \times C_m$, where $(m,2)=1$. 
We know that ${\rm SmallGroup}(32,7)$ has a  maximal subgroup $M\cong C_2 \times D_8$. Therefore
$G$ has a maximal subgroup $H=M \times C_m \cong C_2 \times D_8 \times C_m$.
Using GAP we have
$\psi({\rm SmallGroup}(32,7))=167$ and $\psi(M)=39$ and so $\psi(M) |G:M|^2=39\cdot4=156$.
Now note that
\[\psi(G)=\psi({\rm SmallGroup}(32,7)) \psi(C_m)= 167 \psi(C_m) \nleqq \psi(H)|G:H|^2= 156 \psi(C_m). \]
Therefore this example shows that [HLM]-Conjecture is not satisfied in general. Obviously this is a natural question to ask with which extra conditions the validity of [HLM]-Conjecture is obtained.

\end{document}